\documentclass[12pt]{amsart}
\usepackage{amsfonts,amssymb,amscd,amsmath,enumerate,verbatim,calc}
\usepackage{graphicx}

\newtheorem{Theorem}{Theorem}[section]
\newtheorem{Lemma}[Theorem]{Lemma}
\newtheorem{Corollary}[Theorem]{Corollary}
\newtheorem{Proposition}[Theorem]{Proposition}

\newtheorem{Example}[Theorem]{Example}

\begin{document}
\title{On The Fixatic Number of Graphs}

%

\author{Muhammad Fazil, Imran Javaid}
\keywords{fixing number, fixatic partition, fixatic number
\indent $^*$ Corresponding author: mfazil@bzu.edu.pk}
\address{Centre for advanced studies in Pure and Applied Mathematics,
Bahauddin Zakariya University Multan, Pakistan, Email:
{mfazil@bzu.edu.pk, imran.javaid@bzu.edu.pk}}

\date{}
\maketitle
\begin{abstract} The fixing number of a graph $G$ is the smallest cardinality of a
set of vertices $F\subseteq V(G)$ such that only the trivial
automorphism of $G$ fixes every vertex in $F$. Let $\Pi$ $=$
$\{F_1,F_2,\ldots,F_k\}$ be an ordered $k$-partition of $V(G)$. Then
$\Pi$ is called a {\it fixatic partition} if for all $i$; $1\leq
i\leq k$, $F_i$ is a fixing set for $G$. The cardinality of a
largest fixatic partition is called the {\it fixatic number} of $G$.
In this paper, we study the fixatic numbers of graphs. Sharp bounds
for the fixatic number of graphs in general and exact values with
specified conditions are given. Some realizable results are also
given in this paper.

\end{abstract}

\section{Introduction}
Let $G$ be a graph with vertex set $V(G)$ and edge set $E(G)$. The
{\it neighborhood} of a vertex $v$ of $G$ is the set $N(v) = \{u\in
V(G)\ :\ uv\in E(G)\}$. The number of vertices in $N(v)$ is the {\it
degree} of $v$, denoted by $d(v)$. A vertex of degree one is called
a \emph{leaf} or a \emph{pendant vertex}. If two distinct vertices
$u$ and $v$ of $G$ have the property that $N(u)- \{v\} = N(v)-
\{u\}$, then $u$ and $v$ are called {\it twin vertices} (or simply
twins) in $G$. If for a vertex $u$ of $G$, there exists a vertex
$v\neq u$ in $G$ such that $u, v$ are twins in $G$, then $u$ is said
to be a {\it twin} in $G$. A set $T\subseteq V(G)$ is said to be a
{\it twin-set} in $G$ if any two of its elements are twins.

An \textit{automorphism} $\alpha$ of $G$, $\alpha: V(G)\rightarrow
V(G),$ is a bijective mapping such that $\alpha(u)\alpha(v)\in E(G)
$ if and only if $uv \in E(G).$ Thus, each automorphism $\alpha$ of
$G$ is a permutation of the vertex set $V(G)$ which preserves
adjacencies and non-adjacencies. The \textit{automorphism group} of
a graph $G$, denoted by $\Gamma(G)$, is the set of all automorphisms
of a graph $G$. A connected graph $G$ is \emph{symmetric} if
$\Gamma(G)\neq \{id\}$. The \emph{stabilizer} of a vertex $v\in
V(G)$, denoted $\Gamma_{v}(G)$, is the set $\{\alpha \in \Gamma(G) :
\alpha(v)=v\}$. The stabilizer of a set of vertices $F\subseteq
V(G)$ is $\Gamma_{F}(G)=\{\alpha\in \Gamma(G) : \alpha(v)=v,
\forall\ v\in F\}$. Note that $\Gamma_{F}(G)=\bigcap_{v \in F}
\Gamma_{v}(G)$. For a vertex $v$ of a graph $G$, the \emph{orbit} of
$v$, denoted $\mathcal{O}(v)$, is the set of all vertices $\{u \in
V(G) : \alpha(v)=u \,\,\mbox {for\,\, some}\,\, \alpha\in
\Gamma(G)\}$. Two vertices $u$ and $v$ are \emph{similar} if they
belong to the same orbit. 

A vertex $v$ is \textit{fixed} by a group element $g \in \Gamma(G)$
if $g\in \Gamma_{v}(G)$. A set of vertices $F \subseteq V (G)$ is a
\textit{fixing set} of $G$ if $\Gamma_{F}(G)$ is trivial. The
\textit{fixing number} of a graph $G$, denoted by $fix(G)$, is the
smallest cardinality of a fixing set of $G$ \cite{Erw}. The graphs
with $fix(G) = 0$ are the \emph{rigid graphs} \cite{alb}, which have
trivial automorphism group. Every graph $G$ has a fixing set (the
set of vertices of $G$ itself). It is also clear that, any set
containing all but one vertex is a fixing set. Thus, for a graph $G$
on $n$ vertices, $0\leq fix(G)\leq n-1$ \cite{alb}.

The fixing number of a graph $G$ was defined by Erwin and Harary
\cite{Erw} for the first time. Boutin introduced the concept of
determining set and defined it as follows: A subset $D$ of vertices
in a graph $G$ is called a \emph{determining set} if whenever
$g,h\in \Gamma(G)$ with the property that $g(u) = h(u)$ for all
$u\in D$, then $g(v) = h(v)$ for all $v\in V(G)$ \cite{bou}. The
minimum cardinality of a determining set is called the
\emph{determining number}. In \cite{Gib}, it was shown that fixing
set and determining set are equivalent. A considerable literature
has been developed in this field (see \cite{cac, Gib, Har, slt}).
The concept of fixing number originates from the idea of breaking
symmetries in graphs, which have applications in the problem of
programming a robot to manipulate objects \cite{ly}.

Given an ordered set $W=\{w_{1}, w_{2},...,w_{k}\}$ of vertices of a
connected graph $G$ and a vertex $v$ of $G$, the locating code of
$v$ with respect to $W$ is $(d(v,w_{1}),d(v,w_{2})...,d(v,w_{k}))$,
denoted by $c_{W}(v)$. The set $W$ is called a \emph{locating set}
for $G$ if distinct vertices have distinct codes.  The
\emph{location number} of a graph $G$, denoted by $loc(G)$, is the
minimum cardinality of a locating set for $G$. This notion was
introduced by Slater in \cite{slt}. Independently, Harary and Melter
\cite{mel} studied this notion and used the term \emph{metric
dimension} rather than the location number.

A partition of vertex set of a graph is the partition of the
vertices into subsets with some specific conditions. In graph
theory, a number of partitions of the vertices of a graph have been
introduced. In 1977, Cockayne and Hedetnieme \cite{Coc} introduced
the concept of a \emph{domatic partition} in which each subset is a
dominating set of a graph. In 2000, Chartrand, Salehi and  Zhang
\cite{csz} defined a {\it resolving partition} which corresponds to
resolving sets of a graph. In 2002, Chartrand, Erwin, Henning,
Slater and P. Zhang \cite{cehsz} introduced the concept of a
\emph{locating- chromatic partition} in which every class is a color
class. In 2013, Salman, Javaid and Chaudhary \cite{sim} defined a
{\it locatic partition} in which each subset is a locating set. The
maximum number of classes in a locatic partition is called the {\it
locatic number} of $G$ and is denoted by $\mathcal{L}(G)$.

In this paper, we define the fixatic partition and the corresponding
fixatic number of any connected graph $G$ of order $n$ as follows:\\
Let $\Pi$ $=$ $\{F_1,F_2,\ldots,F_k\}$ be an ordered $k$-partition
of $V(G)$. Then $\Pi$ is called a {\it fixatic partition} if for all
$i$; $1\leq i\leq k$, $F_i$ is a fixing set for $G$. The {\it
fixatic number} of $G$, denoted by $F_{xt}(G)$, is the maximum
number of classes in a fixatic partition.  We use $\Pi_{t}$ to
denote the number of possible fixatic partitions having $F_{xt}(G)$
number of classes.




%


Unless otherwise specified, all the graphs $G$ considered in this
paper are simple, connected and symmetric.

This paper is organized as follows. Section 2 provides the study of
fixatic number of graphs and join graphs. Some bounds on fixatic
number with some certain conditions are also given. Section 3
establishes the connections between the fixing and fixatic number of
graphs in the form of realizable results.



\section{Fixatic Number of Graphs}
We begin this section with an example to explain the concept of this
new parameter, that is, we will discuss the technique to find the
fixatic partition and the corresponding fixatic number of a
connected graph $G$.

\begin{Example}\label{exp1l}(co-rising sun graph)

\begin{figure}[h]
        \centerline
        {\includegraphics[width=4cm]{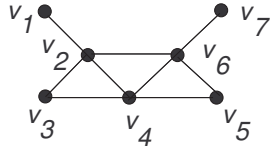}}
        \caption{The graph with $fix(G)= 1.$}\label{f1}
\end{figure}
From the graph $G$ of Figure \ref{f1}, we note that, each vertex of
$G$ forms a fixing set except $v_{4}$, so $F_{xt}(G)=6$. Total
number of fixatic partitions $\Pi_{t}$ of $V(G)$ into $F_{xt}$
classes is $6$.
\end{Example}
Every class in a fixatic partition of $V(G)$ is a fixing set (not
necessarily minimum fixing set) for $G$. Let $\Pi =
\{F_1,F_2,\ldots,F_k\}$ be a fixatic partition of $V(G)$ and
$fix(G)$ is the fixing number of $G$, then $|F_i|\geq fix(G)$ for
all $i$; $1\leq i \leq k$. Each $F_{i}$ in $\Pi$ will be referred to
as fixatic class. Note that, if $fix(G)> \lfloor\frac{n}{2}\rfloor$
for all $i$; $1\leq i\leq k$, then a fixatic partition $\Pi$ of
$V(G)$ will have only one class. However, $F_{xt}(G) = 1$, does not
imply that $fix(G) \ngtr \lfloor\frac{n}{2}\rfloor$. For example,
$F_{xt}(K_{1,3})=1$ though $fix(G) = 2= \lfloor\frac{n}{2}\rfloor$.

Note that, if $F_{xt}(G)=n$, then there exists a unique fixatic
partition $\Pi$ of $V(G)$ that contains $n$ disjoint classes each of
cardinality one, and hence $|\Gamma_v(G)|=1$ for all $v\in V(G)$.
Again, if $|\Gamma_v(G)|=1$ for all $v\in V(G)$, then $fix(G)=1$,
and hence $F_{xt}(G)=n$. Thus, we have the following proposition:
\begin{Proposition} Let $G$ be a connected graph, then $F_{xt}(G)=n$
if and only if the stabilizer of each of its vertex is trivial.
\end{Proposition}
An automorphism $\beta \in\Gamma_{V(G)-\{u,v\}}(G)$
\emph{interchanges} two vertices $u$ and $v$ of a graph $G$ if
$\beta(u)=v$ and $\beta(v)=u$.

\begin{Lemma}\label{mmm} Let $u,v\in V(G)$ and $\beta \in \Gamma (G)$
interchanges $u$ and $v$, then $F_{xt}(G)\geq 2$ if and only if
there exist two fixatic classes $F_i$ and $F_j$, $i\neq j$ in any
fixatic partition $\Pi$ such that $u\in F_i$ and $v\in F_j$.
\end{Lemma}

\begin{proof}Suppose that, $F_{xt}(G)\geq 2$. Then each fixatic
partition $\Pi$ contains atleast two fixatic classes. Since $\beta
\in \Gamma (G)$ interchanges $u$ and $v$ so these vertices must
belong to two disjoint fixatic classes for otherwise if they belong
to the same fixatic class say $F_k$, then $F_{xt}(G)\geq 2$ implies
that $F_l$ $(l\neq k)$ is not a fixing set for $G$. Conversely,
suppose that, $F_{xt}(G)< 2$. Then there is a unique fixatic
partition $\Pi$ that contains all the vertices of $G$, which is
contradiction.
\end{proof}

We can find the total number of surjections $g$ from an $n$-set to a
$k$-set by using the following theorem.
\begin{Theorem}{\em \cite{biggs}}\label{thh}
Let $S_r$ denotes the set of surjections from an $n$-set to a
$k$-set, then $|S_r| = k!S(n,k)$, where $S(n,k)$ denotes the number
of partitions of an $n$-set into $k$ classes.
\end{Theorem}

The number of surjections from an $n$-set to a $k$-set
$\{x_1,x_2,\ldots,x_k\}$, with the property that $n_1$ objects go
into the first set $x_1$, $n_2$ go into $x_2$ and so on, is a
multinomial number ${n \choose n_1,n_2,\ldots,n_k}$. We denote, the
number of partitions of an $n$-set into $k$ classes with each class
of cardinality $i$ by $S(n, k(i))$. \par The following proposition
provides the fixatic number and total number of fixatic partitions
of cycle $C_{n}, n\geq 3$ on odd number of vertices by using the
Theorem \ref{thh}:
\begin{Proposition}Let $C_{n}$ be a cycle on $n\geq 3$ odd number of
vertices, then $F_{xt}(C_n)= \lfloor \frac{n}{2}\rfloor$ and
$\Pi_{t}= \frac{1}{(\lfloor \frac{n}{2}\rfloor - 1)!}{n\choose
3}{n-3 \choose \underbrace{2,\ldots,2}_{(\lfloor \frac{n}{2}\rfloor
- 1)\mbox{-times}}}.$
\end{Proposition}
\begin{proof} Note that $fix(C_{n})=2$ and
all pair of vertices form a fixing set. Hence, $F_{xt}(C_n)=\lfloor
\frac{n}{2}\rfloor$.\\ Any partition of $V(G)$ of cardinality
$\lfloor \frac{n}{2}\rfloor$ contains $\lfloor \frac{n}{2}\rfloor-1$
classes of order two and 1 class of order three. The class having
three vertices can be find in ${n\choose 3}$ different possible ways
and $\lfloor \frac{n}{2}\rfloor-1$ classes of order two in $S(n-3,
(\lfloor \frac{n}{2}\rfloor-1)(2))$ different possible ways: the
number of partitions of an $(n-3)$- set, say $\{v_{1},..., v_i,
v_{i+4},..., v_{n}\}$ into $(\lfloor \frac{n}{2}\rfloor-1)$ classes each of order 2.\\
For this, we find the number of surjections from an $(n-3)$-set,
$U=\{v_{1},..., v_i, v_{i+4},..., v_{n}\}$ to a $(\lfloor
\frac{n}{2}\rfloor-1)$-set $X=\{x_1,x_2,...,x_k\}$ which distributes
the vertices of $U$ into the sets of $X$, in such a way that, each
set $x_{i}, 1\leq i\leq k$ receives exactly two vertices from $U$.
Let $S_r$ denotes the set of surjections from an $(n-3)$-set $U$ to
a $(\lfloor \frac{n}{2}\rfloor-1)$-set $X$, with the property that,
two vertices go into the first set $x_1$, two go into $x_2$, and so
on, then $|S_{r}| = {n-3 \choose \underbrace{2,\ldots,2}_{(\lfloor
\frac{n}{2}\rfloor - 1)\mbox{-times}}}.$\\ By Theorem \ref{thh}, we
have $|S_r|=S(n-3,(\lfloor \frac{n}{2}\rfloor-1)(2))\cdot(\lfloor
\frac{n}{2}\rfloor-1)!$ which implies that, $S(n-3,(\lfloor
\frac{n}{2}\rfloor-1)(2))=\frac{|S_{r}|}{(\lfloor
\frac{n}{2}\rfloor-1)!}.$ Hence, $\Pi_{t}= \frac{1}{(\lfloor
\frac{n}{2}\rfloor - 1)!}{n\choose 3}|S_{r}|.$
\end{proof}

%

Sharp upper and lower bounds for fixatic number of a connected graph
in terms of the fixing and location numbers, are given in the
following lemma:

\begin{Lemma}\label{lemf1} Let $G$ be a connected graph of order $n\geq
2$ and $\mathcal{L}(G)$ be the locatic number of $G$, then
$\mathcal{L}(G)\leq F_{xt}(G)\leq \lfloor\frac{n}{fix(G)}\rfloor.$
Both bounds are sharp.
\end{Lemma}

\begin{proof} 
For $F_{xt}(G)= k$, $n=\sum \limits_{i=1}^{k}|F_{i}|\geq kfix(G)$
which implies that $F_{xt}(G)\leq \lfloor\frac{n}{fix(G)}\rfloor.$
Lower bound follows from the fact that $fix(G)\leq loc(G)$.\\ For
the sharpness of the upper bound, take $G=P_{2n}, n\geq 1$,
we have a single fixatic partition $\Pi$ having $n$ maximum number of classes.\\
For the sharpness of the lower bound, take $G=K_{n}, n\geq 2$, then
$fix(G)=n-1=loc(G)$, and hence we have a single fixatic partition
$\Pi$ having one class that consists of all the vertices of a graph
$G$.
\end{proof}

Since $fix(G)=fix(\overline{G})$, so by Lemma \ref{lemf1}, we have
the following corollary:
\begin{Corollary} Let $G$ be a connected graph of order $n\geq 2$,
then $2 \leq F_{xt}(G)+F_{xt}(\overline{G})\leq 2n.$ Both bounds are
sharp.
\end{Corollary}

A vertex $v$ of a graph $G$ is called \emph{saturated} if it is
adjacent to all other vertices of $G$. Let $e$ be an edge of a
connected graph $G$. If $G-e$ is disconnected, then $e$ is called a
\emph{bridge} or a \emph{cut-edge} of $G$.

\noindent Following result shows, another upper bound for the
fixatic number of a connected graph $G$ by using the definition of a
saturated vertex.
\begin{Proposition}Let $G$ be a connected symmetric graph of order $n\geq 2$, let
$n^{'}$ be the number of its saturated vertices, then $F_{xt}(G)\leq
n-n^{'}+2$. This bound is sharp.
\end{Proposition}
\begin{proof}Suppose that, $F_{xt}(G)> n-n^{'}+2$. From Lemma
\ref {lemf1}, $n-n^{'}+2<F_{xt}(G)\leq \frac{n}{fix(G)}$. Note that,
$n-n^{'}+2\neq 0$ because $n^{'}\leq n$. Thus, $fix(G)<
\frac{n}{n-n^{'}+2}$. Now, if $G$ is connected graph of order $n\geq
4$ with two saturated vertices, then $fix(G)< 1$ in other words
$fix(G)=0$, which is contradiction because $G$ is symmetric.\\For
the sharpness of this bound, take $G= K_{2}$. Hence, $F_{xt}(G)\leq
n-n^{'}+2$.
\end{proof}

\begin{Proposition} Let $T$ be a twin set of a connected graph $G$
with $|T|\geq 3$, then $F_{xt}(G)\leq F_{xt}(G-u)$, for all $u\in
T$.
\end{Proposition}
\begin{proof}Suppose that, $u\in T$. Then $G-u$ is an induced subgraph of $G$.
Since $|T|\geq 3$, so $fix(G-u)< fix(G)$. Since $fix(G)< n$, so
$\frac{n}{fix(G-u)}> \frac{n}{fix(G)}$. Thus, by Lemma \ref{lemf1},
$F_{xt}(G)\leq F_{xt}(G-u)$, for all $u\in T$.
\end{proof}

Following this Proposition, we establish the following theorem:

\begin{Theorem} Let $T$ be a twin set of a connected graph $G$
with $|T|\geq 3$, then $F_{xt}(G)\leq F_{xt}(G-B)$ for every subset
$B\subseteq T$ with $|B|\leq|T|-2$.
\end{Theorem}


\begin{Lemma}\label{ttt} Let $G$ be a connected graph of order $n\geq 2$,
then $n\leq fix(G)+F_{xt}(G)\leq n+1.$ Both bounds are sharp.
\end{Lemma}

\begin{proof} Whenever $fix(G)=1$, then $F_{xt}(G)\leq n$, and hence $fix(G)+F_{xt}(G)\leq n+1$.
Note that $F_{xt}(G)=1$ whenever $\lfloor \frac{n}{2}\rfloor <
fix(G)\leq n-1$. Take $fix(G)=n-1$, then $n\leq fix(G)+ F_{xt}(G)$. Hence, $n\leq fix(G)+F_{xt}(G)\leq n+1.$\\
The upper bound is sharp whenever $G=P_{2k}, k\geq 1$ and the lower
bound is sharp whenever $G=K_{n}, n\geq 3$. Hence, $n\leq
fix(G)+F_{xt}(G)\leq n+1.$
\end{proof}

A \emph{fixing vertex} is a vertex which forms a fixing set of a
connected graph $G$.\\Let $G$ be a connected graph and $y\in V(G)$
be a fixing vertex. Now, if $G-y$ is a connected symmetric graph,
then $fix(G)= fix(G-y)$, and hence we have the following result:

\begin{Proposition} Let $G$ be a connected graph of order $n\geq 3$,
and $y\in V(G)$ a fixing vertex. Let $G-y$ be a connected symmetric
graph, then $F_{xt}(G)\geq F_{xt}(G-y)$. This bound is sharp.
\end{Proposition}

\noindent Suppose that $G_{1}=(V_{1},E_{1})$ and
$G_{2}=(V_{2},E_{2})$ be two graphs with disjoint vertex sets
$V_{1}$ and $V_{2}$ and disjoint edge sets $E_{1}$ and $E_{2}$. The
\emph{union} of $G_{1}$ and $G_{2}$ is the graph $G_{1}\cup
G_{2}=(V_{1}\cup V_{2}, E_{1}\cup E_{2})$. The \emph{join} of
$G_{1}$ and $G_{2}$ is the graph $G_{1}+G_{2}$ that consists of
$G_{1}\cup G_{2}$ and all edges joining all vertices of $V_{1}$ with
all vertices of $V_{2}$.

\begin{Theorem}\label{df} For any two connected graphs $G_{1}$ and $G_{2}$,
$fix(G_{1}+G_{2})\geq fix(G_{1})+fix(G_{2})$. This bound is sharp.
\end{Theorem}
\begin{proof} Suppose that, $fix(G_{1})=n_{1}$, $fix(G_{2})=n_{2}$
and $fix(G_{1}+G_{2})=n_{3}$. Let
$F_{3}=\{v_{1},v_{2},...,v_{n_{3}}\}$ be any fixing set of
$G_{1}+G_{2}$ of minimum cardinality. Suppose to the contrary, that
$n_{3}< n_{1}+n_{2}$. Without loss of generality, take $n_{3}=
n_{1}+n_{2}-1$, that is, for any  $v_{i}\in F_{3}, 1\leq i\leq
n_3$, $F_{4}= F_{3}-\{v_{i}\}$ is a fixing set of $G_{1}+G_{2}$.\\
If $v_{i}$ belongs to any fixing set of minimum cardinality of
$G_{1}$ (or of $G_{2}$), then there exists a vertex $v_{j}\in
V(G_{1}), j\neq i$ or $(\in V(G_{2}))$ such that $v_{j}\in
\mathcal{O}(v_{i})$, and hence $F_{4}$ is not a fixing set of
$G_{1}+G_{2}$.\\ If $v_{i}$ does not belong to any fixing set of
minimum cardinality of $G_{1}$ and of $G_{2}$, then $v_{i}\in
V(G_{1}+G_{2})$ is a saturated vertex. Now, if $v_{i}\in V(G_{1})$,
then there must exists a saturated vertex $v_j\in V(G_2)$, such that
$\Gamma_{\{v_i, v_j\}}(G_1+G_2)$ is non trivial. Thus, $F_4$ is not
a fixing set of $G_1+G_2$, which is contradiction. Hence,
$fix(G_{1}+G_{2})> fix(G_{1})+fix(G_{2})$. Sharpness of this bound
follows, when we take, $G_{1}= K_{2}$ and $G_{2}= C_{n}, n\geq 4$.
Hence, $fix(G_{1}+G_{2})\geq fix(G_{1})+fix(G_{2})$.
\end{proof}
Now, we establish the sharp upper bound of a join graph by using
Theorem \ref{df}.
\begin{Theorem} Let $G_{1}$ and $G_{2}$ be the connected graphs of
order $n_{1}, n_{2}\geq 2$, respectively and $G=G_{1}+G_{2}$ be the
join graph of order $n_{3}\geq 5$, then $F_{xt}(G)\leq
min\{F_{xt}(G_{1}), F_{xt}(G_{2})\}.$ This bound is sharp.
\end{Theorem}
\begin{proof} Suppose that, $k= F_{xt}(G)> F_{xt}(G_{1})$, where
$F_{xt}(G_{1})=min\{F_{xt}(G_{1}),F_{xt}(G_{2})\}.$ Then there will
$k$ classes consisting of fixing sets of $G$, that is, which will
fix vertices of $G_{1}$ as well as $G_{2}$. This means vertices of
$G_{1}$ will be partitioned into more classes than $F_{xt}(G)$ which
will $fix(G)$ but this is contradiction with the definition of
$F_{xt}(G)$. For the sharpness of this bound, take $G_{1}= K_{2}$
and $G_{2}= P_{2n}, n\geq 2$. Hence, $F_{xt}(G)\leq
min\{F_{xt}(G_{1}), F_{xt}(G_{2})\}.$
\end{proof}

\begin{Proposition} Let $G_{1} = K_{1}$ and $G_{2}$ be any connected graph and $G=G_{1}+G_{2}$
with $|G_{2}|\geq 2$, then $F_{xt}(G)\leq F_{xt}(G_{2})$. This bound
is sharp.
\end{Proposition}

\section{Some Realizable results}
In the following lemma, we show that every integer $t\geq 2$ is
realizable as the fixing number and the fixatic number of some
connected graphs:

\begin{Lemma} For any integer $t\geq 2$, there exists a
connected graph $G$ such that $fix(G)=t=F_{xt}(G)$.
\end{Lemma}
\begin{proof}Consider $G$ be a connected graph obtained by a path $P_{2}:uv$ by
joining $2$ paths $P_{t-1}^{i}:
u_{i,1}u_{i,2}...u_{i,t-1}$, $1\leq i\leq 2$ each of order $t-1$
with vertex $u$ and $t$ paths $P_{t-1}^{j}:
v_{j,1}v_{j,2}...v_{j,t-1}$, $1\leq j\leq t$ each of order $t-1$
with vertex $v$ of a path $P_{2}$ respectively. 
Let $A_{i}=\{u_{i,1}, u_{i,2},..., u_{i,t-1}\}$, $1\leq i\leq 2$ and
$B_{j}=\{v_{j,1}, v_{j,2},..., v_{j,t-1}\}$, $1\leq j\leq t$. Note
that, any minimum fixing set $F$ of $G$ must contains exactly one
vertex from either $A_{1}$ or $A_{2}$ and one vertex from all
$B_{j}$, $1\leq j\leq t$, except one. This implies that $fix(G)=t$.
We also observe that, the sets $B^{\ast}_{j}=\{v_{i,j}\}\cup \{u_{i=
1,j}\}, 1\leq j\leq t-1,\,\,1\leq i\neq j\leq t$ and
$B^{\ast}=\{v_{i,j}, u_{2,2}\}, 1\leq i=j\leq t-1$ are the maximum
number of fixing sets of minimum cardinality of $G$. Hence, each
fixatic partition $\Pi$ contains $t-1$ classes of cardinality $t$,
and one class of cardinality $2t$ because $t-2$ vertices $\{u_{i,j},
i\neq j\}$ of $P_{t-1}^{i= 2}$ and two vertices of $P_{2}$ do not
form a fixing set. It follows that $F_{xt}(G)=t$.
\end{proof}
\begin{Proposition} Let $G$ be a connected graph. If
$fix(G)=F_{xt}(G)$, then no class of any fixatic partition $\Pi$ of
$V(G)$ can be a singleton set.
\end{Proposition}
\begin{proof} Contrary suppose that, a partition $\Pi_{j}$ contains a
singleton class, $F_{j}= \{v\}$, $v\in V(G)$. Note that, $v$ is a
fixing vertex, and hence $F_{xt}(G)=1$, which implies that
$\Gamma(G)=\{e\}$, a contradiction.
\end{proof}


A \emph{rooted tree} is a tree with a labeled vertex called a
\emph{root}.

\begin{Lemma}\label{dan1} For all integers $t\geq 2$, there exists a
connected graph $G$ such that $fix(G)=t$ and $F_{xt}(G)=t+1$.
\end{Lemma}
\begin{proof}Consider a rooted tree with $v$ as its root obtained by joining $t+1$ paths
$P_{t}^{i}: v_{i,1}v_{i,2}...v_{i,t}$, $1\leq i\leq t+1$
each of order $t$ with the root. 
Note that, the minimum fixing set for $G$ contains exactly one
vertex from each path $P_{t}^{i}$, $1\leq i\leq t$, which implies
that $fix(G)=t$. We also observe that, the sets $A_{j}=\{v_{i,j}\},
1\leq j\leq t,\,\,1\leq i\neq j\leq t+1$ and $A=\{v_{i,j}\}, 1\leq
i=j\leq t$ are the maximum number of disjoint fixing sets of minimum
cardinality. Hence, each fixatic partition $\Pi$ contains $t$
classes of cardinality $t$ and one class of cardinality $t+1$. It
follows that $F_{xt}(G)=t+1$.
\end{proof}

\begin{Lemma} For any integer $t\geq 3$, there exists a connected graph
$G$ such that $fix(G)+F_{xt}(G)=t$.
\end{Lemma}
\begin{proof} Consider a path $P_{t-2}: v_{1}, v_{2},...,v_{t-2}$.
Join a set $\{v'_{i}, v''_{i}\},\,\ 1\leq i\leq t-2$ of leaves with
each $v_{i},\,\ 1\leq i\leq t-2$ of the path $P_{t-2}$. Note that,
the graph $G$ has $t-2$ twin set of vertices, and hence
$fix(G)=t-2$. We also observe that, each fixatic partition $\Pi$
contains two maximum number of classes, one class of order $t-2$ and
the other class of order $2t-4$. Hence, $F_{xt}(G)=2$ and the result
follows.
\end{proof}

A vertex $v$ of degree at least three in a connected graph $G$ is
called a \emph{major vertex}. Two paths rooted from the same major
vertex and having the same length are called the \emph{twin stems}.

\begin{Lemma}\label{rslt} For any integer $t\geq 1$, there exists a connected graph $G$
such that $F_{xt}(G)-fix(G)=t.$
\end{Lemma}
\begin{proof}Consider a rooted tree as $G$ with $v$ as its root obtained by joining
$t+1$ paths $P_{t}^{i}: v_{i,1}v_{i,2}...v_{i,2t-1}$, $1\leq i\leq
t+1$ each of order $2t-1$ with the root. Since all $t+1$ paths are
in fact the twin stems, $fix(G)=t$. Further $A=\{v_{i,j}\},\,\ 1\leq
i=j\leq t$, $B_{j}=\{v_{i,j}\}, 1\leq j\leq t,\,\,1\leq i\neq j\leq
t+1$, $C_{k}=\{v_{i,j}\},\,\ 1\leq k\leq t-1,\,\ j= k+t,\,\,1\leq
i\leq t$ are the disjoint fixing sets of minimum cardinality. Each
fixatic partition $\Pi$ contains $2t$ maximum number of classes in
such a way that, $2t-1$ classes are of order $t$ and one class of
order $2t$. Hence, $F_{xt}(G)=2t$.
\end{proof}
Lemma \ref{rslt}, can also be stated as follows:
\begin{Lemma} For any integer $t\geq 1$, there exists a connected graph
$G$ such that $fix(G)=t$ and $F_{xt}(G)=2t$.
\end{Lemma}
Consider a path $P_{2}=uv$. For all integers $t\geq 3$, construct a
connected graph $G$ as follows: If $t$ is even, then join
$\frac{t+4}{2}$ leaves with $u$ and  $\frac{t+2}{2}$ leaves with $v$
of the path $P_{2}.$ If $t$ is odd, then join
$\lceil\frac{t}{2}\rceil+1$ leaves with $u$ and same number of
leaves with vertex $v$.\\ From this type of construction of $G$, we
have a useful result, stated as follows:
\begin{Lemma}\label{rslt11} For any integer $t\geq 3$, there exists a connected graph $G$
such that $fix(G)-F_{xt}(G)=t.$
\end{Lemma}


Let $G$ be a connected graph. A \emph{matching} $M$ in a graph $G$
is a set of edges such that no two edges have a common vertex. The
\emph{size} of matching is the number of edges in it. A vertex
contained in an edge of $M$ is \emph{covered} by $M$. A matching
that covers every vertex of the graph $G$ is called a \emph{perfect
matching} \cite{chris}. Note that, each element in a matching is an
edge which is two vertices subset of the vertex set $V(G)$.


\begin{Proposition} Let $G$ be a connected graph of order $n\geq 4$ with
$F_{xt}(G)=2$, then each partite of every fixatic partition $\Pi$ of
$V(G)$ induces an edge if and only if $G\in \{C_{4}, K_{2,2},
K_{1}+P_{3}, K_{4}-e \}$, where $K_{4}-e$ be the graph obtained by
deleting any one edge from $K_{4}$.
\end{Proposition}

Let $\Psi=(H_{n})_{n\geq 1}$ be a family of graphs $H_n$ of order
 $\varphi(n)$ for which $lim_{n\rightarrow
 \infty}\varphi(n)=\infty$. If there does not exist a constant number $c>0$
 such that $F_{xt}(H_n)\leq c$ for every $n\geq 1$, then we say that $\Psi$ has
 \emph{unbounded fixatic number}.

\begin{Example} Let $G=Cay(Z_n; S)$ be an undirected Cayley graph with vertex
set $V(G)=\{v_0,v_1,...,v_{n-1}\}$ and $S$ is its connection set.
Let $n=2p+1$ is prime. We know that, $fix(Cay(Z_n; S))=2$ (every
fixing set of $Cay(Z_n; S)$ must contain at least two vertices from
$V(G)$). So, to form a fixatic partition of maximum cardinality, we
divide the vertex set $V(G)$ into $p-1$ classes of order two, and
one class of order three, otherwise all the classes may not be
fixing. So, there will be $(p-1)+1$ classes, which form a fixatic
partition. Hence, $F_{xt}({Cay(Z_n; S)})=p$, which implies that
$G=Cay(Z_n; S)$ is a family of graphs with unbounded fixatic number.
\end{Example}

A family $\Omega=(G_{n})_{n\geq 1}$ of connected graphs is a family
with \emph{constant fixing number} if there exists a constant
$0<c<\infty$ such that $fix(G_{n})=c,$ for all $n$. The family of
paths $P_{n}, n\geq 2$ is a family with constant fixing number,
because $fix(P_{n})=1$.

\begin{Theorem} A family of connected graphs with constant fixing
number is a family of graphs with unbounded fixatic number.
\end{Theorem}
\begin{proof}Let $\Omega=(G_{n})_{n\geq 1}$ be a family of graphs
with constant fixing number $fix(G_{n})$ and fixatic number
$F_{xt}(G_{n})$, then there exists a constant $c>0$ such that
$fix(G_{n})=c$ for all $n$. In order to form a fixatic partition
$\Pi=\{F_{1}, F_{2},...,F_{F_{xt}}\}$ of $V(G_{n})$, each $F_{j}$
has $|F_{j}|\geq c$. So, there are at least $\lfloor
\frac{n}{c}\rfloor$ classes in $\Pi$, which implies that there does
not exist a constant $c>0$ such that $F_{xt}(G_{n})\leq c$ for all
$n$.
\end{proof}



\end{document}